\documentclass[11pt,letter]{article}

\usepackage[francais,english]{babel}
\usepackage[latin1]{inputenc}
\usepackage[T1]{fontenc}
\usepackage{amsmath,amssymb}
\usepackage{graphicx}

\def\R{\mathbb{R}}

\newtheorem{defi}{Definition}
\newtheorem{lem}{Lemma}
\newtheorem{prop}{Proposition}

\newtheorem{rem}{Remark}
\newtheorem{theo}{Theorem}

\title{Results and questions on  a nonlinear approximation approach for solving
  high-dimensional partial differential equations}

\author{C. Le Bris$^1$, T. Leli\`evre $^1$ \& Y. Maday $^2$\\
{\footnotesize $^1$ CERMICS, \'Ecole des Ponts,}\\
{\footnotesize 6 \& 8, avenue Blaise Pascal, 77455 Marne-La-Vall\'ee
  Cedex 2, FRANCE, and}\\
{\footnotesize INRIA Rocquencourt, MICMAC project-team,}\\
 {\footnotesize Domaine de Voluceau,  B.P. 105,
 78153 Le Chesnay Cedex, FRANCE}\\
{\footnotesize \tt \{lebris,lelievre\}@cermics.enpc.fr}\\
{\footnotesize $^2$ Laboratoire J.L.-Lions, Universit\'e Pierre et Marie
   Curie,}\\
{\footnotesize Boite courrier 187, F-75252 Paris, FRANCE}\\
{\footnotesize \tt maday@ann.jussieu.fr}
}

\begin{document}

\maketitle

\begin{abstract}
 We  investigate  mathematically a
nonlinear approximation type approach recently introduced
in~\cite{ammar-mokdad-chinesta-keunings-06} to solve  high dimensional
partial differential equations. We show the link between the approach
and the \emph{greedy algorithms}
of approximation theory studied {\it e.g.}
  in~\cite{devore-temlyakov-96}. On the prototypical case of the
  Poisson equation, we show that a variational version of the approach,
   based on minimization of energies, converges. On the other
  hand, we show various theoretical and numerical difficulties arising
  with the non variational version of the approach, consisting of simply
  solving the first order optimality equations of the problem. Several unsolved issues
  are indicated in order to motivate further research.
\end{abstract}

\section{Introduction}

Our purpose here is to investigate  mathematically a
numerical approach recently introduced
in~\cite{ammar-mokdad-chinesta-keunings-06} to solve  high dimensional
partial differential equations.

The approach is a nonlinear approximation type approach that consists in 
expanding the solution of the equation in
tensor products of functions sequentially determined as the iterations
of the algorithm proceed. The original motivation of the approach is
the wish of its authors to solve high-dimensional Fokker-Planck type
equations arising in the modelling of complex fluids. Reportedly, the
approach performs well in this case, and, in addition, extends 
to a large variety of
partial differential equations, static or time-dependent, linear or
nonlinear, elliptic or parabolic, involving self-adjoint or non self-
adjoint operators provided the data enjoy
some appropriate separation property  with respect to the different
coordinates (this property is made precise in Remark~\ref{rem:RHS} below). We refer the reader to~\cite{ammar-mokdad-chinesta-keunings-06} for more
details.

In the present contribution focused on mathematical analysis, we restrict ourselves to the simplest
possible case, namely the solution of the Poisson equation set with
Dirichlet homogeneous
boundary conditions  on a two dimensional parallelepipedic domain  $\Omega=\Omega_x
\times \Omega_y$ with $\Omega_x \subset
\R$ and $\Omega_y \subset
\R$  bounded.
In short, the approach under consideration then determines the solution $u$ to
\begin{equation}
\label{eq:11}
-\Delta u(x,y)=f(x,y)
\end{equation}
as a sum
\begin{equation}
\label{eq:22}
u(x,y)=\sum_{n\geq 1}r_n(x)\,s_n(y),
\end{equation}
by iteratively determining functions $r_n(x)$, $s_n(y)$, $n\geq 1$ such
that for all $n$, $r_n(x)\,s_n(y)$ is the best approximation (in a sense
to be made precise below) of the solution $v(x,y)$ to
$ -\Delta
v(x,y)=f(x,y)+\Delta \left(\sum_{k\leq n-1}r_k(x)\,s_k(y)\right)
$ in terms of one single
tensor product $r(x)s(y)$.
We show that it is possible to give  a sound mathematical
ground to the approach \emph{provided} we consider a variational form of the approach that
manipulates minimizers of energies instead of solutions to equations. In 
order to reformulate the approach in
such a variational setting, our arguments thus crucially exploit the
fact that the Laplace operator is self-adjoint. It is to be already
emphasized that, because of the nonlinearity of
the tensor product expansion (\ref{eq:22}),  the variational form of the approach is \emph{not} equivalent
to the form (\ref{eq:11})-(\ref{eq:22}) (which is exactly the Euler-Lagrange equations associated to the energy considered in the variational approach). Our analysis therefore does
not apply to the actual implementation of the method as described in~\cite{ammar-mokdad-chinesta-keunings-06}. At present time,
we do not know how to extend our arguments to cover the practical
situation, even in the simple  case of the Poisson problem. The
consideration of some particular pathological cases, theoretically and
numerically, shows that the appropriate mathematical setting is
unclear.  Likewise, it is unclear to us how to provide a mathematical
foundation of the approach for non variational situations, such as an equation involving a  differential operator that is not self-adjoint.

On the other hand, the analysis provided here
straightforwardly extends to the case of a $N$-dimensional
 Poisson problem with $N\geq 3$ (unless explicitly mentioned). Likewise,
 our analysis extends to the case of  elliptic linear partial differential equations set on a cylinder in
$\R^N$, with appropriate boundary conditions. The only, although substantial, difficulty that may appear when the dimension $N$ grows is the algorithmic complexity of the
approach, since a set of $N$ coupled non-linear equations has to be solved (see Remark~\ref{rem:HD}). At least, the number of unknowns involved in the systems to be solved 
does not grow exponentially, as it would be the case for a naive approach (like for a finite differences method on a tensorized grid).  This is
not the purpose of the present article to further elaborate
on this.

\medskip

Our article is organized as follows. Section~\ref{sec:presentation} introduces the approach. The  variational version of the approach
(along with a relaxed variant of it) is described in
Section~\ref{sec:algo}. Elementary properties follow in Sections~\ref{sec:well} and \ref{sec:EL}. The non variational version is presented in
Section~\ref{sec:nonvar}. In
Section~\ref{sec:convergence} we show the convergence of the variational
approach and give an estimate of the rate of convergence. Our arguments
immediately follow from standard arguments  of the literature of \emph{nonlinear
approximation theory}, and especially from those of~\cite{devore-temlyakov-96}. The particular approach under consideration
is indeed closely related  to the so-called \emph{greedy algorithms}
introduced in approximation theory. We refer
to~\cite{barron-cohen-dahmen-devore-08,davis-mallat-avellaneda-97,temlyakov-08}
for some relevant contributions, among many. The  purpose of
Section~\ref{sec:discussion} is to return to the original non
variational formulation of the approach. For  illustration,
we first consider the case when the Laplace operator $-\Delta$ in
(\ref{eq:11}) is replaced by the identity operator. The approach then
reduces to the determination of the \emph{Singular Value Decomposition} (also called \emph{rank-one decomposition}) of the
right-hand side $f$. This simple situation allows one to understand
various difficulties inherent to the non variational formulation of the approach. We
then discuss the actual case of the Laplace operator, and present some
intriguing numerical experiments, in particular when a non-symmetric
term (namely there an advection term) is added.

\medskip

As will be clear from the sequel, our current mathematical understanding 
of the
numerical approach is rather incomplete. Our results do not cover real
practice. Some ingredients from the  literature of nonlinear approximation
theory nevertheless already allow for understanding some basics of the
approach. It is the hope of the authors that, laying some groundwork, the present contribution
will sparkle some interest among the experts, and allows in a not too
far future for a
complete understanding of the mathematical nature of the approach. Should the need
arise, it will also indicate possible improvements of the approach so
that it is rigorously founded  mathematically and, eventually,
performs even  better that the currently existing reports seemingly show.

\medskip

{\bf Acknowledgments}: The authors wish to thank A.~Ammar and F.~Chinesta for introducing them
to their series of works initiated
in~\cite{ammar-mokdad-chinesta-keunings-06},  A.~Cohen for stimulating
discussions, and
A.~Lozinski for pointing out reference~\cite{devore-temlyakov-96}. This work was
completed while the first author (CLB) was a long-term visitor at the
Institute for Mathematics and its Applications, Minneapolis. The
hospitality of this institution is gratefully acknowledged.


\section{Presentation of the algorithms}
\label{sec:presentation}

Consider a function $f \in L^2(\Omega)$ where $\Omega=\Omega_x
\times \Omega_y$ with $\Omega_x \subset
\R$ and $\Omega_y \subset
\R$ two bounded domains. To fix ideas, one may take
$\Omega_x=\Omega_y=(0,1)$. Consider on $\Omega$ the following homogeneous
Dirichlet problem:
\begin{equation}\label{eq:lapl}
\text{Find $g \in H^1_0(\Omega) $ such that }\left\{
\begin{array}{rl}
-\Delta g = f &\text{ in $\Omega$},\\
g=0& \text{ on $\partial \Omega$}.
\end{array}
\right.
\end{equation}
It is well known that solving~\eqref{eq:lapl} is equivalent to solving the variational problem:
\begin{equation}\label{eq:lapl_var}
\text{Find $g \in H^1_0(\Omega)$ such that } g=\arg\min_{u \in H^1_0(\Omega)} \left(
  \frac{1}{2}\int_\Omega |\nabla u|^2 - \int_\Omega f u \right).
\end{equation}
In the following, for any function $u \in H^1_0(\Omega)$, we denote
\begin{equation}\label{eq:ener}
{\mathcal E}(u)=\frac{1}{2}\int_\Omega |\nabla u|^2 - \int_\Omega f u.
\end{equation}
Notice that
\begin{equation}\label{eq:ener_polar}
{\mathcal E}(u)=\frac{1}{2}\int_\Omega |\nabla (u - g)|^2 -
\frac{1}{2}\int_\Omega |\nabla g|^2
\end{equation}
where $g$ is defined by~\eqref{eq:lapl}, so that minimizing ${\mathcal
  E}$ is equivalent to minimizing $\int_\Omega |\nabla (u - g)|^2$ with
respect to $u$. We endow the functional space
$H^1_0(\Omega)$ with the scalar product:
$$\langle u , v \rangle = \int_\Omega \nabla u \cdot  \nabla v,$$
and the associated norm
$$\| u \|^2= \langle u , u \rangle =  \int_\Omega |\nabla u|^2.$$

\subsection{Two algorithms}
\label{sec:algo}

We now introduce two algorithms to solve~\eqref{eq:lapl}. The first algorithm is the \emph{Pure Greedy Algorithm}: {\tt set $f_0=f$, and at iteration
$n \ge 1$,
\begin{enumerate}
\item Find $r_n \in H^1_0(\Omega_x)$ and $s_n \in H^1_0(\Omega_y)$ such
  that
\begin{equation}\label{eq:LRA_var}
(r_n,s_n)=\arg \min_{(r,s) \in H^1_0(\Omega_x) \times H^1_0(\Omega_y)} 
\left(
  \frac{1}{2}\int_\Omega |\nabla (r \otimes s)|^2 - \int_\Omega f_{n-1}  
\,r \otimes s \right).
\end{equation}
\item Set $f_{n}=f_{n-1} + \Delta (r_n \otimes s_n)$.
\item If $\|f_{n}\|_{H^{-1}(\Omega)} \ge \varepsilon$, proceed to iteration $n+1$. Otherwise, stop.
\end{enumerate}
}
Throughout this article, we denote by $r \otimes s$ the tensor product: $r \otimes s(x,y)=r(x)
s(y)$. Notice that $$f_n=f + \Delta \left( \sum_{k=1}^{n} r_k \otimes s_k \right).$$
The fonction $f_n$ belongs to $H^{-1}(\Omega)$ and the tensor product $r 
\otimes s$
is in $H^1_0(\Omega)$ if $r \in H^1_0(\Omega_x)$ and $s \in
H^1_0(\Omega_y)$ (see Lemma~\ref{lem:H10} below), so that the integral $\int_\Omega f_{n-1}  \,r \otimes s$
in~\eqref{eq:LRA_var} is well defined.

A variant of this algorithm is the \emph{Orthogonal Greedy Algorithm}:
{\tt set $f_0^o=f$, and at iteration
$n \ge 1$,
\begin{enumerate}
\item Find $r_n^o \in H^1_0(\Omega_x)$ and $s_n^o \in H^1_0(\Omega_y)$ such
  that
\begin{equation}\label{eq:LRA_varo}
(r_n^o,s_n^o)=\arg\min_{(r,s) \in H^1_0(\Omega_x) \times H^1_0(\Omega_y)} \left(
  \frac{1}{2}\int_\Omega |\nabla (r \otimes s)|^2 - \int_\Omega f_{n-1}^o  \,r \otimes s \right).
\end{equation}
\item Solve the following Galerkin problem on the basis $(r_1^o \otimes
  s_1^o, \ldots, r_n^o \otimes  s_n^o )$: find $(\alpha_1, \ldots,
  \alpha_n) \in \R^n$ such that
\begin{equation}\label{eq:LRA_galo}
(\alpha_1, \ldots, \alpha_n)=\arg\min_{(\beta_1, \ldots,
  \beta_n) \in \R^n} \left(
  \frac{1}{2}\int_\Omega \left|\nabla \left( \sum_{k=1}^n \beta_k r_k^o \otimes
  s_k^o \right) \right|^2 - \int_\Omega f \,\sum_{k=1}^n \beta_k r_k^o 
\otimes
  s_k^o \right) .
\end{equation}
\item Set $f_{n}^o=f + \Delta \left( \sum_{k=1}^n \alpha_k r_k^o \otimes s_k^o \right)$.
\item If $\|f_{n}^o\|_{H^{-1}(\Omega)} \ge \varepsilon$, proceed to iteration $n+1$. Otherwise, stop.
\end{enumerate}
}
Let us also introduce $g_n$ satisfying the Dirichlet problem:
\begin{equation}\label{eq:g_n}
\left\{
\begin{array}{rl}
-\Delta g_n = f_n & \text{ in $\Omega$},\\
g_n=0 & \text{ on $\partial \Omega$}.
\end{array}
\right.
\end{equation}
Notice that
\begin{equation}\label{eq:g_np1}
g_{n}=g_{n-1} - r_n \otimes s_n.
\end{equation}
so that $g_n=g - \sum_{k=1}^{n} r_k \otimes s_k$.
Likewise, we introduce $g_n^o=g - \sum_{k=1}^{n} r_k^o \otimes s_k^o$,
which satisfies $-\Delta g_n^o = f_n^o$ in $\Omega$ and $g_n^o=0$ on
$\partial \Omega$. Proving the convergence of the algorithms amounts to proving that $g_n$ and $g_n^o$ converge to $0$.

The terminology \emph{Pure Greedy Algorithm} and \emph{Orthogonal Greedy
  Algorithm} is borrowed from approximation theory
(see~\cite{barron-cohen-dahmen-devore-08,davis-mallat-avellaneda-97,devore-temlyakov-96,temlyakov-08}).
Such algorithms have been introduced in a more general
framework, namely for an arbitrary Hilbert space and an arbitrary set of functions (not
only tensor products). Recall for consistency that, in short, the purpose of such nonlinear
approximations techniques is to find the best possible approximation of  
a given function as a  sum of
elements of a prescribed \emph{dictionary}. The latter does not need to
have a vectorial structure. In the present case, the dictionary is the
set of simple products $r(x)s(y)$ for $r$ varying in
$H^1_0(\Omega_x)$ and $s$ varying in
$H^1_0(\Omega_y)$ (All this will be formalized with the introduction of
the space ${\mathcal L}^1$  in Section~\ref{sec:convergence} below). The
metric chosen to define the approximation is the natural metric
induced by the differential operator, here the $H^1$ norm. The algorithm 
proposed by Ammar {\em et
  al}.~\cite{ammar-mokdad-chinesta-keunings-06} is actually  related to
the Orthogonal Greedy Algorithm: it consists in replacing the
optimization procedure~\eqref{eq:LRA_varo} by the associated
Euler-Lagrange equations. We shall give details on this in
Section~\ref{sec:EL} below. For the moment, we concentrate ourselves on
the variational algorithms above.

\subsection{The iterations are well defined}
\label{sec:well}

We will need the following three lemmas.

\begin{lem}\label{lem:H10}
For any measurable functions $r : \Omega_x \to \R$ and $s : \Omega_y \to
\R$ such that $r \otimes s \neq 0$
$$ r \otimes s \in H^1_0(\Omega) \iff r \in H^1_0(\Omega_x) \text{ and } 
s \in H^1_0(\Omega_y).$$
\end{lem}

\begin{lem}\label{lem:distrib}
Let $T \in {\mathcal D}'(\Omega)$ be a distribution such that, for any
functions $(\phi,\psi) \in {\mathcal C}^\infty_c(\Omega_x) \times {\mathcal C}^\infty_c(\Omega_y)$,
$$(T, \phi \otimes \psi)_{({\mathcal D}'(\Omega),{\mathcal
    D}(\Omega))}=0$$
then $T=0$ in ${\mathcal D}'(\Omega)$. Moreover, for any two sequences of distributions $R_n \in {\mathcal D}'(\Omega_x)$ and $S_n\in {\mathcal D}'(\Omega_y)$   such that $\lim_{n \to \infty} R_n = R$ in ${\mathcal D}'(\Omega_x)$ and $\lim_{n \to \infty} S_n=S$ in ${\mathcal D}'(\Omega_y)$, $\lim_{n \to \infty} R_n \otimes S_n = R 
\otimes S$ in ${\mathcal D}'(\Omega)$.
\end{lem}

\begin{lem}\label{lem:E_neg}
Let us consider a function $f \in L^2(\Omega)$. If $f \neq 0$, then $\exists (r,s) \in H^1_0(\Omega_x) \times H^1_0(\Omega_y)$ such that
$${\mathcal E}(r \otimes s)< 0,$$
where ${\mathcal E}$ is defined by~\eqref{eq:ener}.
\end{lem}

Lemma~\ref{lem:distrib} is well-known in distribution theory. We now provide
for consistency a short proof of Lemmas~\ref{lem:H10} and
\ref{lem:E_neg}, respectively.

\medskip

\noindent{\bf Proof of Lemma~\ref{lem:H10}}
Notice that
$$\int_\Omega |\nabla (r \otimes s) |^2=\int_{\Omega_x} |r '|^2 \int_{\Omega_y} |s|^2 + \int_{\Omega_x}
|r|^2 \int_{\Omega_y} |s'|^2$$
where $'$ denotes henceforth the differentiation with respect to a
one-dimensional argument. Thus, it is clear that if $r \in
H^1_0(\Omega_x)$ and $s \in H^1_0(\Omega_y)$, then $r \otimes s \in
H^1_0(\Omega)$. Now, when $r \otimes s \in
H^1_0(\Omega)$, we have  $\int_{\Omega_x} |r '|^2 \int_{\Omega_y}
|s|^2 < \infty$ and $\int_{\Omega_x}
|r|^2 \int_{\Omega_y} |s'|^2 < \infty$. This implies  $r \in
H^1_0(\Omega_x)$ and $s \in H^1_0(\Omega_y)$, since $r \neq 0$ and $s
\neq 0$.
\hfill$\diamondsuit$

\medskip

\noindent{\bf Proof of Lemma~\ref{lem:E_neg}}
Fix $f \in L^2(\Omega)$ and assume that for all $(r,s) \in
H^1_0(\Omega_x) \times H^1_0(\Omega_y)$, ${\mathcal E}(r \otimes s)\ge
0$. Then, for a fixed $(r,s) \in H^1_0(\Omega_x) \times
H^1_0(\Omega_y)$, we have, for all $\epsilon \in \R$,
$$\frac{\epsilon^2}{2} \int | \nabla (r \otimes s) |^2 \ge \epsilon \int 
f r \otimes
s.$$
By letting $\epsilon \to 0$, this shows that $f \in \{r \otimes s, \, (r,s)\in
L^2(\Omega_x)\times L^2(\Omega_y)\}^\perp$ which implies $f=0$ (by Lemma~\ref{lem:distrib}) and concludes the proof.
\hfill$\diamondsuit$

\medskip

The above lemmas allow us to prove.
\begin{prop}\label{prop:LRA_var_well_posed}
For each $n$, there exists a solution to problems~\eqref{eq:LRA_var} and~\eqref{eq:LRA_varo}.
\end{prop}
\begin{proof}
Without loss of generality, we may only argue on problem~\eqref{eq:LRA_var} and assume that $n=1$ and $f_0=f \neq 0$.
First, using~\eqref{eq:ener_polar}, it is clear that
\begin{align*}
\frac{1}{2}\int_\Omega |\nabla (r \otimes s)|^2 - \int_\Omega f  \,r \otimes s
&= \frac{1}{2}\int_\Omega |\nabla (r
\otimes s - g)|^2 -
\frac{1}{2}\int_\Omega |\nabla g|^2\\
&\ge - \frac{1}{2}\int_\Omega |\nabla g|^2.
\end{align*}
Thus, we can introduce $m=\inf_{(r,s) \in H^1_0(\Omega_x) \times H^1_0(\Omega_y)} \left(
  \frac{1}{2}\int_\Omega |\nabla (r \otimes s)|^2 - \int_\Omega f  \,r
  \otimes s \right)$ and a minimizing sequence $(r^k,s^k)$ such that
$\lim_{k \to \infty} {\mathcal E}(r^k \otimes s^k) = m$. Notice that we
may suppose, again  without loss of generality (up to a multiplication of $s^k$
by a constant), that
$$\int_\Omega |r^k|^2 = 1.$$

Since ${\mathcal E}(u) \ge \frac{1}{4} \int_\Omega |\nabla u|^2 -
 \int_\Omega |\nabla g|^2$,
the sequence $(r^k \otimes s^k)$ is bounded in $H^1_0(\Omega)$: there
exists some $C >0$ such that, for all $k \ge 1$,
\begin{equation}\label{eq:subseq_bounded}
\int_{\Omega_x} |(r^k) '|^2 \int_{\Omega_y} |s^k|^2 + \int_{\Omega_x}
|r^k|^2 \int_{\Omega_y} |(s^k)'|^2\le C.
\end{equation}
From this we deduce the existence of  $w \in H^1_0(\Omega)$, $r \in
L^2(\Omega_x)$ and $s \in H^1_0(\Omega_y)$ such that (up to the
extraction of a subsequence):
\begin{itemize}
\item $r^k \otimes s^k$ converges to $w$ weakly in $H^1_0(\Omega)$, and
  strongly in $L^2(\Omega)$,
\item $r^k$ converges to $r$ weakly in $L^2(\Omega_x)$,
\item $s^k$ converges to $s$ weakly in $H^1_0(\Omega_y)$, and
  strongly in $L^2(\Omega_y)$.
\end{itemize}
Since $r^k \otimes s^k$ converges to $w$ weakly in $H^1_0(\Omega)$ and
${\mathcal E}$ is convex and continuous, we have ${\mathcal E}(w) \le
\liminf_{k \to \infty} {\mathcal E} (r^k \otimes s^k)$. This yields
${\mathcal E}(w) \le m$. Moreover, by Lemma~\ref{lem:E_neg}, we know $m<0$. Therefore,
\begin{equation}\label{eq:w_neg}
{\mathcal E}(w) < 0.
\end{equation}

The convergences $r^k \to r$ and $s^k \to s$ in the distributional sense
imply the convergence $r^k \otimes s^k \to r \otimes s$ in the distributional
sense (see Lemma~\ref{lem:distrib}), and therefore $w = r \otimes s$. Thus, if $w
\neq 0$, Lemma~\ref{lem:H10} concludes the proof,  showing that indeed
$r \in H^1_0(\Omega_x)$. Now,  we cannot have $w=0$, since this would imply
${\mathcal E}(w) = 0$, which would contradict~\eqref{eq:w_neg}. This
concludes the proof.
\end{proof}


The optimization step~\eqref{eq:LRA_galo} admits also a solution by standard arguments and we therefore have proven:
\begin{lem}\label{lem:algo_well_defined}
At each iteration of the Pure Greedy Algorithm, problem
(\ref{eq:LRA_var}) admits (at least) a minimizer $(r_n,s_n)$. Likewise,
at each iteration of the Orthogonal Greedy Algorithm, problem
(\ref{eq:LRA_varo}) admits (at least) a minimizer $(r_n^o,s_n^o)$.
\end{lem}

  It is important to note that, in either case,  uniqueness of the
  iterate is unclear. Throughout the text, we will thus be refering to
  {\it the} functions $(r_n,s_n)$ (resp. $(r_n^o,s_n^o)$) although we
  do not know whether they are unique. However, our arguments and
  results are valid for \emph{any  such} functions. 

\subsection{Euler-Lagrange equations}
\label{sec:EL}

Our purpose is now to derive the  Euler-Lagrange equations of the problems
considered,  along with other important properties of the sequences
$(r_n,s_n)$ and $(r_n^o,s_n^o)$. We only state the results for
$(r_n,s_n)$. Similar  properties hold for  $(r_n^o,s_n^o)$, replacing $f_n$ and $g_n$ by $f_n^o$ and $g_n^o$.

The first order  optimality conditions write:
\begin{prop}
The functions $(r_n,s_n) \in H^1_0(\Omega_x)\times H^1_0(\Omega_y)$
satisfying~\eqref{eq:LRA_var} are such that: for any
functions $(r,s) \in H^1_0(\Omega_x)\times H^1_0(\Omega_y)$
\begin{equation}\label{eq:LRA_EL_FV}
\int_{\Omega} \nabla (r_n \otimes s_n) \cdot \nabla (r_n \otimes s + r \otimes
s_n) =  \int_{\Omega} f_{n-1} (r_n \otimes s + r \otimes s_n).
\end{equation}
This  can be written equivalently as
\begin{equation}\label{eq:LRA_EL}
\left\{
\begin{array}{l}
\displaystyle - \left(\int_{\Omega_y} |s_n|^2\right)\,  r_n''  + \left(\int_{\Omega_y}
|s_n'|^2\right)\, =  r_n  \int_{\Omega_y} f_{n-1}\, s_n,\\
\\
\displaystyle - \left(\int_{\Omega_x} |r_n|^2 \right)\,  s_n''  + \left(\int_{\Omega_x}
|r_n'|^2\right)\,  = s_n  \int_{\Omega_x} f_{n-1}\, r_n,
\end{array}
\right.
\end{equation}
or, in terms of $g_n$, as: 
\begin{equation}\label{eq:ortho}
\langle g_{n} , (r \otimes s_n + r_n \otimes s) \rangle = 0.
\end{equation}
\end{prop}
\begin{proof}
Equation~\eqref{eq:LRA_EL_FV} is obtained  differentiating~\eqref{eq:LRA_var}. Namely, for any $(r,s) \in H^1_0(\Omega_x)\times H^1_0(\Omega_y)$ and any $\epsilon \in \R$, we have
\begin{eqnarray}
\label{eq:eps}
&&\frac{1}{2}\int_\Omega |\nabla \left( (r_n + \epsilon r) \otimes (s_n +
\epsilon s) \right)|^2 - \int_\Omega f_{n-1}  \, (r_n + \epsilon r)
\otimes (s_n + \epsilon s)\nonumber\\
&&\quad \quad \quad \quad \ge\frac{1}{2}\int_\Omega |\nabla (r_n
\otimes s_n) |^2 - \int_\Omega f_{n-1}  \, r_n \otimes s_n.
\end{eqnarray}
It holds:
\begin{align*}
\frac{1}{2}&\int_\Omega |\nabla \left(  (r_n + \epsilon r) \otimes (s_n +
\epsilon s) \right)|^2 - \int_\Omega f_{n-1}  \,(r_n + \epsilon r)
\otimes (s_n + \epsilon s)\\
&=\frac{1}{2}\int_\Omega |\nabla (r_n \otimes s_n) + \epsilon \nabla (r
\otimes s_n + r_n \otimes s) + \epsilon^2 \nabla (r \otimes s)|^2 - \int_\Omega f_{n-1}  \,(r_n + \epsilon r)
\otimes (s_n + \epsilon s)\\
&=\frac{1}{2}\int_\Omega |\nabla (r_n
\otimes s_n)|^2 - \int_\Omega f_{n-1}  \, r_n \otimes s_n \\
&\quad + \epsilon \left( \int_\Omega \nabla( r_n \otimes s_n )\cdot \nabla (r
  \otimes s_n + r_n \otimes s) - \int_\Omega f_{n-1} (r_n \otimes s + r
  \otimes s_n )\, \right)\\
& \quad + \epsilon^2 \left(\frac{1}{2}\int_\Omega |\nabla (r
  \otimes s_n + r_n \otimes s)|^2 + \int_\Omega \nabla (r_n \otimes s_n) 
\cdot \nabla (r \otimes s) - \int_\Omega f_{n-1} r \otimes s \right)
+ O(\epsilon^3)\\
&= \frac{1}{2}\int_\Omega |\nabla (r_n
\otimes s_n)|^2 - \int_\Omega f_{n-1}  \, r_n \otimes s_n + \epsilon I_1 
+ \epsilon^2 I_2 +  O(\epsilon^3).
\end{align*}
Using~\eqref{eq:eps}, we get, for any $\epsilon \in \R$,
\begin{equation}\label{eq:I1I2}
\epsilon I_1 + \epsilon^2 I_2 +  O(\epsilon^3) \ge 0,
\end{equation}
which implies that $I_1$ is zero, that is, \eqref{eq:LRA_EL_FV}.

Equation~\eqref{eq:LRA_EL} is the strong formulation
of~\eqref{eq:LRA_EL_FV}. On the other hand, \eqref{eq:ortho} is an
immediate  consequence of the following simple computations:
\begin{align*}
\langle g_{n} , (r \otimes s_n + r_n \otimes s) \rangle
&= \langle g_{n-1} - r_n \otimes s_n , (r \otimes s_n + r_n \otimes s)
\rangle\\
&= \int_\Omega \nabla(g_{n-1} - r_n \otimes s_n) \cdot \nabla(r \otimes s_n +
r_n \otimes s)\\
&= - \int_\Omega \Delta g_{n-1} (r \otimes s_n +
r_n \otimes s) -\int_\Omega \nabla (r_n \otimes s_n) \cdot  \nabla(r \otimes s_n +
r_n \otimes s)\\
&=0,
\end{align*}
since $-\Delta g_{n-1} = f_{n-1}$ in $\Omega$ and $g_{n-1}=0$ on $\partial \Omega$.
\end{proof}

\medskip

Note that, taking $r=r_n$ and $s=0$ in  the Euler-Lagrange
equations~\eqref{eq:ortho} yields
\begin{equation}\label{eq:ortho2}
\langle r_n \otimes s_n , g_{n-1} \rangle = \|r_n \otimes s_n\|^2,
\end{equation}
since $g_n= g_{n-1} - r_n \otimes s_n$. This will be useful below.

\medskip

Let us now state two other properties of $(r_n,s_n)$. The second order optimality conditions write:
\begin{lem}\label{lem:EL2}
The functions $(r_n,s_n) \in H^1_0(\Omega_x)\times H^1_0(\Omega_y)$
satisfying~\eqref{eq:LRA_var} are such that: for any
functions $(r,s) \in H^1_0(\Omega_x)\times H^1_0(\Omega_y)$
\begin{equation}\label{eq:LRA_EL2_FV}
\frac{1}{2}\int_\Omega |\nabla (r
  \otimes s_n + r_n \otimes s)|^2 + \int_\Omega \nabla (r_n
  \otimes s_n) \cdot \nabla (r \otimes s) - \int_\Omega f_{n-1} r \otimes s \ge 0,
\end{equation}
which is equivalent to: for any
functions $(r,s) \in H^1_0(\Omega_x)\times H^1_0(\Omega_y)$
\begin{equation}\label{eq:LRA_EL2_FV'}
\left(\int_\Omega \nabla ( r_n \otimes s_n - g_{n}) \cdot \nabla (r \otimes s) \right)^2 \le \int_\Omega |\nabla (r
  \otimes s_n)|^2\ \int_\Omega |\nabla (r_n \otimes s)|^2.
\end{equation}
\end{lem}
\begin{proof}
Returning to Equation~\eqref{eq:I1I2}, we see that $I_1=0$ and $I_2 \ge
0$, which is exactly~\eqref{eq:LRA_EL2_FV}. For any $\lambda
\in \R$, taking $(\lambda r, s)$ as a test function
in~\eqref{eq:LRA_EL2_FV} shows
$$\frac{1}{2}\int_\Omega | \lambda \nabla ( r
  \otimes s_n) + \nabla ( r_n \otimes s)|^2 +   \int_\Omega \lambda \nabla (r_n
  \otimes s_n) \cdot \nabla ( r \otimes s) - \int_\Omega f_{n-1} \lambda 
r \otimes s
  \ge 0.$$
This equivalently reads
\begin{align*}
\frac{\lambda^2}{2} &\int_\Omega |\nabla (r  \otimes s_n)|^2 + \lambda
\left(\int_\Omega \left( \nabla (r
  \otimes s_n) \cdot \nabla (r_n \otimes s) + \nabla (r_n
  \otimes s_n) \cdot \nabla (r \otimes s )\right)  - \int_\Omega f_{n-1} 
r
  \otimes s \right) \\
& + \frac{1}{2}\int_\Omega |\nabla (r_n  \otimes s)|^2
\ge 0,
\end{align*}
hence
\begin{align*}
&\left(\int_\Omega \left( \nabla (r
  \otimes s_n) \cdot \nabla (r_n \otimes s) + \nabla (r_n
  \otimes s_n) \cdot \nabla (r \otimes s) \right) -
  \int_\Omega f_{n-1} r \otimes s \right)^2 \\
 &\quad \quad \le \int_\Omega |\nabla (r
  \otimes s_n)|^2\ \int_\Omega |\nabla (r_n \otimes s)|^2.
\end{align*}
This yields~\eqref{eq:LRA_EL2_FV'}.
\end{proof}

We will also need the following optimality property of $(r_n,s_n)$:
\begin{lem}\label{lem:ProdScal}
The functions $(r_n,s_n)$ satisfying~\eqref{eq:LRA_var} are such that: $\forall (r,s) \in H^1_0(\Omega_x)\times H^1_0(\Omega_y)$
$$\|r_n \otimes s_n\|=\frac{\langle r_n \otimes s_n , g_{n-1} \rangle}{\|r_n \otimes s_n\|}
\ge \frac{\langle r \otimes s , g_{n-1} \rangle}{\|r \otimes s\|}.$$
\end{lem}
\begin{proof}
We may
assume without loss of generality that $n=1$. The first equality is~\eqref{eq:ortho2}. To prove the inequality, let us introduce the supremum:
$$M=\sup_{(u,v) \in  H^1_0(\Omega_x)\times H^1_0(\Omega_y),
  \| u \otimes v \|=1} \langle u \otimes v , g \rangle.$$
Using~\eqref{eq:ortho2}, we have
\begin{equation}\label{eq:3}
\|r_1 \otimes s_1\| = \frac{\langle r_1 \otimes s_1 ,g \rangle}{\|r_1
  \otimes s_1\|} \le M,
\end{equation}
by definition of $M$. On the other hand,  consider
$(u^k,v^k)_{k \ge 0}$ a maximizing sequence associated to the supremum~$M$: $\| u^k \otimes v^k \|=1$ and $\lim_{k \to \infty}
\langle u^k \otimes v^k , g \rangle= M$.
We have, using~\eqref{eq:LRA_var}, for all  $k \ge 0$,
\begin{align*}
\|g-r_1 \otimes s_1\|^2
& \le \| g - \langle g , u^k \otimes v^k \rangle \, u^k \otimes v^k  \|^2\\
&= \| g \|^2 - \langle g , u^k \otimes v^k \rangle^2,
\end{align*}
and, letting $k \to \infty$,
\begin{equation}\label{eq:4}
\|g-r_1 \otimes s_1\|^2 \le \| g \|^2 - M^2.
\end{equation}
Combining~\eqref{eq:3} and~\eqref{eq:4}, we get
\begin{align*}
\|g-r_1 \otimes s_1\|^2
& \le \| g \|^2 - M^2 \\
& \le \| g \|^2 - \|r_1 \otimes s_1 \|^2 \\
&  = \| g \|^2 - 2 \langle g , r_1 \otimes s_1 \rangle + \|r_1 \otimes
s_1\|^2\\
&= \| g - r_1 \otimes s_1\|^2
\end{align*}
so that all the inequalities are actually equalities. By using the fact that, by~\eqref{eq:3}, $M \ge 0$, we thus have
$$M=\|r_1 \otimes s_1 \|=\frac{\langle r_1 \otimes s_1 ,g \rangle}{\|r_1
  \otimes s_1\|}.$$
This concludes the
proof.
\end{proof}

\medskip

\subsection{Some preliminary remarks on  the non variational
  approach implemented}
\label{sec:nonvar}
Before we get to the proof of the convergence of the approach in the
next section, let us conclude Section~\ref{sec:presentation} by some comments that
relates the theoretical framework developed here to the practice.

\medskip

It is important to already note, although we will return to this in
Section~\ref{sec:discussion} below, that the Euler-Lagrange equation is
indeed the form of the algorithm manipulated in practice by the authors
of~\cite{ammar-mokdad-chinesta-keunings-06}. The above variational setting is somewhat difficult to implement in
practice. It requires to solve for the minimizers of \eqref{eq:LRA_var}
(and \eqref{eq:LRA_varo} respectively), which can
be
extremely demanding
computationally.   In their implementation of the approach (developed
independently from the above nonlinear approximation theoretic
framework), Ammar {\em et al.} therefore propose to search for the iterate
$(r_n,s_n)$ (and respectively $(r_n^o,s_n^o)$) not as a minimizer to
optimization problems~\eqref{eq:LRA_var} and~\eqref{eq:LRA_varo}, but as
a solution to the associated Euler-Lagrange equations (first order optimality conditions).
The Pure Greedy algorithm is thus replaced by:
set $f_0=f$, and at iteration $n \ge 1$,
{\tt
\begin{enumerate}
\item Find $r_n \in H^1_0(\Omega_x)$ and $s_n \in H^1_0(\Omega_y)$ such
  that, for all functions $(r,s) \in H^1_0(\Omega_x) \times H^1_0(\Omega_y)$,
~\eqref{eq:LRA_EL_FV} (or its equivalent form \eqref{eq:LRA_EL})  holds.
\item Set $f_{n}=f_{n-1} +\Delta ( r_n \otimes s_n)$.
\item If $\|f_{n}\|_{H^{-1}(\Omega)} \ge \varepsilon$, proceed to iteration $n+1$. Otherwise, stop.
\end{enumerate}
}
The Orthogonal Greedy Algorithm is modified likewise.

As already explained in the introduction, and in sharp contrast to the
situation encountered for linear problems, being a solution to the
Euler-Lagrange equation does not guarantee being a minimizer in this
nonlinear framework. We will point out difficulties originating from
this in Section~\ref{sec:discussion}.

In addition to the above theoretical difficulty, and in fact somehow
entangled to it, we have to mention that of course,
 the Euler-Lagrange equations~\eqref{eq:LRA_EL}, as a  nonlinear
 system,  need to be solved iteratively.   In~\cite{ammar-mokdad-chinesta-keunings-06}, a simple 
fixed point procedure is employed: choose $(r_n^0,s_n^0) \in H^1_0(\Omega_x)\times H^1_0(\Omega_y)$ and, at iteration $k \ge 0$, compute $(r_n^k,s_n^k) \in H^1_0(\Omega_x)\times H^1_0(\Omega_y)$ solution to:
\begin{equation}\label{eq:LRA_FP}
\left\{
\begin{array}{l}
\displaystyle - \int_{\Omega_y} |s_n^k|^2  (r_n^{k+1})''  + \int_{\Omega_y}
|(s_n^k)'|^2  r_n^{k+1} = \int_{\Omega_y} f_{n-1} s_n^k,\\
\\
\displaystyle - \int_{\Omega_x} |r_n^{k+1}|^2  (s_n^{k+1})''  + \int_{\Omega_x}
|(r_n^{k+1})'|^2  s_n^{k+1} = \int_{\Omega_x} f_{n-1} r_n^{k+1},
\end{array}
\right.
\end{equation}
until convergence is reached.
We will also discuss below the convergence properties of this procedure on simple examples.

\begin{rem}\label{rem:RHS}
In practice (bearing in mind that the approach has been designed to solve 
high-dimensional problems), in order for the right-hand side terms
in~\eqref{eq:LRA_FP} to be computable, the function $f$ needs to be
expressed as a sum of tensor products. Otherwise, computing high
dimensional integrals would be necessary, and this is a task of the same computational
 complexity as the original Poisson problem.  The function $f$ thus 
needs to enjoy some appropriate separation property with respect to the different coordinates.

If $f$ is not given in such a form, it may be possible to first apply the Singular Value Decomposition algorithm to get a good estimate of $f$ as 
a sum of tensor products (see Section~\ref{sec:SVD}).
\end{rem}

\begin{rem}\label{rem:HD}
In dimension $N \ge 2$ (on a parallelepipedic domain $\Omega=\Omega_{x_1} \times \ldots \times \Omega_{x_N}$), the Euler-Lagrange equations~\eqref{eq:LRA_EL_FV} become: find functions $(r^1_n,\ldots,r^N_n) \in 
H^1_0(\Omega_{x_1})\times \ldots \times H^1_0(\Omega_{x_N})$ such that: for any
functions $(r^1,\ldots,r^N) \in H^1_0(\Omega_{x_1})\times \ldots \times H^1_0(\Omega_{x_N})$,
\begin{align}
\int_{\Omega} &\nabla (r^1_n \otimes \ldots \otimes r^N_n) \cdot \sum _{k=1}^N \nabla (r^1_n \otimes \ldots \otimes r^{k-1}_n \otimes r^k \otimes r^{k+1}_n \otimes \ldots \otimes r^N_n) \nonumber\\
& =  \int_{\Omega} f_{n-1} \sum _{k=1}^N (r^1_n \otimes \ldots \otimes r^{k-1}_n \otimes r^k \otimes r^{k+1}_n \otimes \ldots \otimes r^N_n).\label{eq:LRA_EL_FV_HD}
\end{align}
This is a nonlinear system of $N$ equations, which only involves one-dimensional integrals by Fubini theorem, provided that the data $f$ is
expressed as a sum of tensor products (see Remark~\ref{rem:RHS}).
\end{rem}

\begin{rem}\label{rem:disc}
We presented the algorithms without space discretization, which is required for the practical implementation. In practice, finite element spaces $V^h_x$ (resp. $V^h_y$) are used to discretized $H^1_0(\Omega_x)$ (resp. $H^1_0(\Omega_y)$), where $h>0$ denotes a space discretization parameter. The space discretized version of~\eqref{eq:LRA_EL_FV} thus writes: find $(r^h_n,s^h_n) \in V^h_x \times V^h_y$ such that, for any
functions $(r^h,s^h) \in V^h_x \times V^h_y$
\begin{equation}\label{eq:LRA_EL_FV_disc}
\int_{\Omega} \nabla (r^h_n \otimes s^h_n) \cdot \nabla (r^h_n \otimes s^h + r^h \otimes
s^h_n) =  \int_{\Omega} f^h_{n-1} (r^h_n \otimes s^h + r^h \otimes s^h_n).
\end{equation}
\end{rem}

\section{Convergence}
\label{sec:convergence}

To start with, we prove that the approach converges. Then we will turn
to the rate of convergence.
\subsection{Convergence of the method}

\begin{theo}\label{theo:CV_PGA} {\bf [Pure Greedy Algorithm]}

Consider the Pure Greedy Algorithm, and  assume first that $(r_n,s_n)$ satisfies the Euler-Lagrange equations~\eqref{eq:LRA_EL_FV}.
Denote by
\begin{equation}
E_n=\frac{1}{2}\int_\Omega |\nabla (r_n \otimes s_n)|^2 - \int_\Omega f_{n-1}  \,r_n \otimes s_n
\end{equation}
 the energy at iteration $n$.
We have  
\begin{equation}\label{eq:serCV}
\displaystyle\sum_n \int_\Omega |\nabla (r_n \otimes
s_n)|^2  = - 2 \sum_n E_n < \infty.
\end{equation}
Assume in addition that $(r_n,s_n)$ is a minimizer of \eqref{eq:LRA_var}. Then, 
\begin{equation}\label{eq:gCV}
\lim_{n \to \infty} g_n = 0 \text{ in $H^1_0(\Omega)$.}
\end{equation}
\end{theo}
Immediate consequences of~\eqref{eq:serCV} and~\eqref{eq:gCV} are
$$\lim_{n \to \infty} E_n = \lim_{n \to \infty} \| r_n \otimes
s_n\| = 0,$$
 and
$$\lim_{n \to \infty} f_n = 0  \text{ in $H^{-1}(\Omega)$}.$$
\begin{proof}
Let us first suppose that $(r_n,s_n)$ only satisfies the Euler-Lagrange equations~\eqref{eq:LRA_EL_FV}. We notice that, using~\eqref{eq:ortho}
\begin{align}
\|g_{n-1}\|^2
&=\|g_{n} + r_n \otimes s_n\|^2 \nonumber\\
&=\|g_{n}\|^2 + \|r_n \otimes s_n\|^2. \label{eq:pyth}
\end{align}
Thus, $\|g_n\|^2$ is a nonnegative non increasing sequence. Hence it converges. This implies that $\sum_n |\nabla (r_n \otimes
s_n)|^2 < \infty$.

Next, notice that
\begin{align*}
E_n
&= \frac{1}{2}\int_\Omega |\nabla (r_n \otimes s_n)|^2 - \int_\Omega f_{n-1}
\,r_n \otimes s_n\\
&=\frac{1}{2}\int_\Omega |\nabla (r_n \otimes s_n)|^2 - \int_\Omega \nabla
g_{n-1} \cdot
 \nabla (r_n \otimes s_n)\\
&=- \frac{1}{2}\int_\Omega |\nabla (r_n \otimes s_n)|^2,
\end{align*}
since by~\eqref{eq:ortho2}, $\int_\Omega \nabla
g_{n-1} \cdot
 \nabla (r_n \otimes s_n)=\int_\Omega |\nabla (r_n \otimes
 s_n)|^2$. This proves the first part of the theorem. At this stage, we have only used that
$(r_n,s_n)$ satisfies the Euler-Lagrange equations~\eqref{eq:LRA_EL}.

To conclude that $\lim_{n \to \infty} f_n=  0$, we now need to assume that
$(r_n,s_n)$ indeed satisfies the minimization problem~\eqref{eq:LRA_var}. We know
that $\|g_n\|^2$ is a bounded sequence, and therefore, we may assume (up
to the extraction of a subsequence) that $g_n$ converges weakly in
$H^1_0(\Omega)$ to some $g_\infty \in H^1_0(\Omega)$. For any $n \ge 1$ and for
any functions $(r,s) \in H^1_0(\Omega_x) \times H^1_0(\Omega_y)$,
$$\frac{1}{2}\int_\Omega |\nabla (r \otimes s)|^2 - \int_\Omega \nabla g_{n-1}
\cdot \nabla (r
\otimes s )\ge E_n.$$
By passing to the limit this inequality, we have
$$\frac{1}{2}\int_\Omega |\nabla (r \otimes s)|^2 - \int_\Omega \nabla g_\infty \cdot
\nabla (r \otimes s )\ge 0.$$
This implies that for any functions $(r,s) \in H^1_0(\Omega_x) \times
H^1_0(\Omega_y)$,
$$\int_\Omega \nabla g_\infty \cdot \nabla (r \otimes s)=  0.$$
Thus, by Lemma~\ref{lem:distrib}, $-\Delta g_\infty=  0$ in the
distributional sense, which, since $g_\infty \in H^1_0(\Omega)$, implies
$g_\infty=0$. This shows that there is
only one possible
limit for the subsequence $g_n$ and thus that the whole sequence
itself weakly converges to $0$.

The convergence of $g_n$ to $0$ is actually strong in
$H^1_0(\Omega)$. The argument we use here is taken from~\cite{jones-87}. 
Using
Lemma~\ref{lem:ProdScal}, we have: for any $n \ge m \ge 0$
\begin{align*}
\| g_{n} - g_m \|^2
&=\| g_{n} \|^2 +  \| g_m \|^2 - 2 \left\langle g_n , \left(g_n +  \sum_{k=m+1}^n r_k
\otimes s_k\right) \right\rangle\\
&=\| g_{n} \|^2 +  \| g_m \|^2 - 2 \| g_n \|^2 - 2  \sum_{k=m+1}^n
\langle g_n, r_k \otimes s_k  \rangle\\
&\le - \| g_{n} \|^2 +  \| g_m \|^2  + 2  \sum_{k=m+1}^n
\|r_k \otimes s_k\| \|r_{n+1} \otimes s_{n+1}\|.
\end{align*}
Define $\phi(1)=1$,
$\phi(2)=\arg \min_{n>\phi(1)} \{\|r_n \otimes s_n\| \le \|r_{\phi(1)}
\otimes s_{\phi(1)}\|  \}$, and, by induction,
$$\phi(k+1)=\arg \min_{n>\phi(k)} \{\|r_n \otimes s_n\| \le \|r_{\phi(k)}
\otimes s_{\phi(k)}\|  \}.$$
Notice that $\lim_{k \to \infty} \phi(k)=\infty$ since $\lim_{k \to
  \infty} \|r_k \otimes s_k\|=0$.
For example, if $(\|r_k \otimes s_k\|)_{k \ge 1}$ is a decreasing
sequence, then $\phi(k)=k$. Now, we have: for any $l \ge k \ge 0$
\begin{align*}
\| g_{\phi(l)-1} - g_{\phi(k)-1} \|^2
&\le - \| g_{\phi(l)-1} \|^2 +  \| g_{\phi(k)-1} \|^2  + 2
\sum_{i=\phi(k)}^{\phi(l)-1} \|r_i \otimes s_i\| \|r_{\phi(l)} \otimes 
s_{\phi(l)}\|\\
&\le - \| g_{\phi(l)-1} \|^2 +  \| g_{\phi(k)-1} \|^2  + 2
\sum_{i=\phi(k)}^{\phi(l)-1} \|r_i \otimes s_i\|^2.
\end{align*}
Since $\sum_{k \ge 1} \|r_k \otimes s_k\|^2 < \infty$ and $(\|g_n\|)_{n
  \ge 1}$ is converging, the previous inequality shows that the subsequence
$(g_{\phi(l)-1})_{l \ge 0}$ is a Cauchy sequence, and therefore strongly
converges to $0$ (recall it is already known that $g_n$ weakly converges
to $0$). Since $\|g_n\|$ is itself a converging sequence, this shows
that
$$\lim_{n \to \infty} \|g_n\| = 0.$$
\end{proof}

\medskip

A similar result holds for the Orthogonal Greedy Algorithm.
\begin{theo}\label{theo:CV_OGA} {\bf [Orthogonal Greedy Algorithm]}

Consider the Orthogonal Greedy Algorithm, and  assume first that
$(r_n^o,s_n^o)$ only satisfies the Euler-Lagrange
equations~\eqref{eq:LRA_EL_FV} associated with~\eqref{eq:LRA_varo} (thus
with $(r_n,s_n,f_{n-1})=(r_n^o,s_n^o,f_{n-1}^o)$
in~\eqref{eq:LRA_EL_FV}). Denote by
\begin{equation}
E_n^o=\frac{1}{2}\int_\Omega |\nabla (r_n^o \otimes s_n^o)|^2 - \int_\Omega f_{n-1}^o  \,r_n^o \otimes s_n^o
\end{equation}
the energy at iteration $n$. We have  
\begin{equation}\label{eq:serCVo}
\displaystyle\sum_n \int_\Omega |\nabla (r_n^o \otimes
s_n^o)|^2  =- 2 \sum_n E_n^o < \infty.
\end{equation}
Assume in addition that $(r_n^o,s_n^o)$ is indeed a minimizer to  the optimization problem~\eqref{eq:LRA_varo}. Then, 
\begin{equation}\label{eq:gCVo}
\lim_{n \to \infty} g_n^o = 0 \text{ in $H^1_0(\Omega)$.}
\end{equation}
\end{theo}
Immediate consequences of~\eqref{eq:serCVo} and~\eqref{eq:gCVo} are
$$\lim_{n \to \infty} E_n^o = \lim_{n \to \infty} \| r_n^o \otimes
s_n^o\| = 0,$$
 and
$$\lim_{n \to \infty} f_n^o = 0  \text{ in $H^{-1}(\Omega)$}.$$
\begin{proof}
Let us first assume that $(r_n^o,s_n^o)$ only satisfies the Euler-Lagrange equations~\eqref{eq:LRA_EL_FV} (with $(r_n,s_n,f_{n-1})=(r_n^o,s_n^o,f_{n-1}^o)$ in~\eqref{eq:LRA_EL_FV}). Notice that by~\eqref{eq:LRA_galo} 
and~\eqref{eq:ortho2}:
\begin{align*}
\|g_{n}^o\|^2
&=\bigg\|g - \sum_{k=1}^n \alpha_k r_k^o \otimes s_k^o \bigg\|^2\\
&\le \|g_{n-1}^o - r_n^o \otimes s_n^o \|^2\\
&=\|g_{n-1}^o\|^2 - \|r_n^o \otimes s_n^o \|^2.
\end{align*}
Thus, $\|g_{n}^o\|^2$ is a nonnegative non increasing sequence. Hence it converges. This implies that
$\sum_{k \ge 1} \|r_k^o \otimes s_k^o \|^2 < \infty$, and proves the first part of the theorem, using the same arguments as in the proof of Theorem~\ref{theo:CV_PGA}.

Let us now assume in addition that $(r_n^o,s_n^o)$ is a minimizer to \eqref{eq:LRA_varo}. For fixed $r$ and $s$, we 
derive from~\eqref{eq:LRA_varo}:
$$- \frac{1}{2}\int_\Omega |\nabla( r_n^o \otimes s_n^o)|^2  = \frac{1}{2}\int_\Omega |\nabla (r_n^o \otimes s_n^o)|^2 - \int_\Omega f_{n-1}^o
\,r_n^o \otimes s_n^o \le \frac{1}{2}\int_\Omega |\nabla (r \otimes s)|^2
- \int_\Omega f_{n-1}^o \,r \otimes s.$$
Letting $n$ go to infinity, and using the same arguments as in the proof 
of Theorem~\ref{theo:CV_PGA}, this implies that $g_n^o$ weakly converges to $0$ in $H^1_0(\Omega)$.
The proof of the strong convergence of $g_n^o$ to zero is then easy since, using the Euler 
Lagrange equations associated to~\eqref{eq:LRA_galo}:
$$\|g_n^o\|^2 = \langle g_n^o , g \rangle,$$
and the right-hand side converges to $0$.
\end{proof}

\subsection{Rate of convergence of the method}

We now present an estimate of the rate of convergence for both the Pure and the Orthogonal Greedy Algorithms. These results are
borrowed from~\cite{devore-temlyakov-96}. We begin by only citing the result for Pure Greedy Algorithm. On the other hand, with a view
to showing the typical mathematical ingredients at play, we outline the
proof of convergence of the Orthogonal Greedy Algorithm, contained in
the original article~\cite{devore-temlyakov-96}.

\medskip

We first need to introduce a functional space adapted to the convergence
analysis (see~\cite{barron-cohen-dahmen-devore-08,devore-temlyakov-96}).
\begin{defi}
We define the ${\mathcal L}^1$ space as
$${\mathcal L}^1=\left\{g=\sum_{k \ge 0} c_k u_k \otimes v_k,
\text{ where $u_k \in H^1_0(\Omega_x)$, $v_k \in H^1_0(\Omega_y)$, $\|u_k \otimes v_k\|=1$ and  $\sum_{k \ge 0} |c_k| < \infty$} \right\},$$
and we define the ${\mathcal L}^1$-norm as
$$\|g\|_{{\mathcal L}^1}= \inf \left\{ \sum_{k \ge 0} |c_k|, g=\sum_{k \ge 0}
c_k u_k \otimes v_k, \text{ where $\|u_k \otimes v_k\|=1$} \right\},$$
for $g \in {\mathcal L}^1$.
\end{defi}
The following properties may readily be established:
\begin{itemize}
\item The space  ${\mathcal L}^1$ is a Banach space.
\item The space  ${\mathcal L}^1$ is continuously embedded in $H^1_0(\Omega)$.
\end{itemize}
Notice that, in the definition of ${\mathcal L}^1$, the function $g=\sum_{k \ge 0} c_k u_k \otimes v_k$ is indeed well defined in $H^1_0(\Omega)$ as a normally convergent series. This also shows that ${\mathcal L}^1 \subset H^1_0(\Omega)$, and this imbedding is continuous by the triangle inequality $\|\sum_{k \ge 0} c_k u_k \otimes v_k\| \le  \sum_{k \ge 0} |c_k|$.




We do not know if there exists a simple characterization of functions in
${\mathcal L}^1$. Let us however give simple examples of such functions.
\begin{lem}
For any $m>2$, $H^m(\Omega) \cap H^1_0(\Omega) \subset {\mathcal L}^1$.
\end{lem}
\begin{proof}
Without loss of generality, consider the case $\Omega_x=\Omega_y=(0,1)$. Using the fact that $\left\lbrace \phi_k \otimes \phi_l, k,l \ge 1 \right \rbrace$, where $\phi_k(x) =  \sqrt{2} \sin ( k \pi x)$, is an orthonormal basis of $L^2(\Omega)$, we can write any function $g \in L^2(\Omega)$ as the series $g=\sum_{k,l \ge 1} g_{k,l} \phi_k \otimes 
\phi_l$, where $g_{k,l} = \int_{\Omega} g \, \phi_k \otimes \phi_l$. It is well known that $$g \in H^1_0(\Omega) \iff \sum_{k,l \ge 1} |g_{k,l}|^2 (k^2+l^2) < \infty$$ and, more generally, for any $m \ge 1$,
$$g \in H^m(\Omega) \cap H^1_0(\Omega) \iff \sum_{k,l \ge 1} |g_{k,l}|^2 
(k^2+l^2)^m < \infty.$$
On the other hand,
\begin{align*}
\|g\|_{\mathcal L^1}
&=\bigg\|\sum_{k,l \ge 1} g_{k,l} \phi_k \otimes \phi_l\bigg\|_{\mathcal L^1} \\
&=\bigg\|\sum_{k,l \ge 1} g_{k,l} \|  \phi_k \otimes \phi_l\| \frac{\phi_k \otimes \phi_l}{\|  \phi_k \otimes \phi_l\|}\bigg\|_{\mathcal L^1} \\
&\le \sum_{k,l \ge 1} |g_{k,l}| \pi \sqrt{k^2 + l^2},
\end{align*}
since $\|  \phi_k \otimes \phi_l\|=  \pi \sqrt{k^2 + l^2}$.
Thus, by the H\"older inequality, we have, for any $m > 2$, if $g \in H^m(\Omega) \cap H^1_0(\Omega)$,
\begin{align*}
\|g\|_{\mathcal L^1}
&\le \pi \sum_{k,l \ge 1} |g_{k,l}| (k^2 + l^2)^{m/2} (k^2 + l^2)^{(1-m)/2}\\
&\le \pi \left( \sum_{k,l \ge 1} |g_{k,l}|^2 (k^2+l^2)^m \right)^{1/2} \left( \sum_{k,l \ge 1} (k^2+l^2)^{1-m} \right)^{1/2}\\
& < \infty,
\end{align*}
since $\sum_{k,l \ge 1} (k^2+l^2)^{1-m}< \infty$ as soon as $m>2$.
\end{proof}
\begin{rem}\label{rem:L1}
More generally, in dimension $N \ge 2$, the same proof shows that: for any $m>1+N/2$, $H^m(\Omega) \cap H^1_0(\Omega) \subset {\mathcal L}^1$.
\end{rem}

Let us now give the rate of convergence of the Pure Greedy
Algorithm. For the details of the proof, we again refer to~\cite{devore-temlyakov-96}. The proof is based on the fundamental lemma:
\begin{lem}[{\cite[Lemma 3.5]{devore-temlyakov-96}}]\label{lem:DVT}
Let us assume that $g \in {\mathcal L}^1$. Then, for any $n \ge 0$, $g_n
\in {\mathcal L}^1$ and we have:
$$\|r_{n+1} \otimes s_{n+1}\| = \frac{\langle g_{n} , r_{n+1} \otimes s_{n+1} \rangle}{\| r_{n+1}
\otimes s_{n+1} \|} \ge \frac{\|g_n\|^2}{\|g_n\|_{{\mathcal L}^1}}.$$
\end{lem}


The following technical result (easily obtained by induction) is also needed.
\begin{lem}[{\cite[Lemma 3.4]{devore-temlyakov-96}}]\label{lem:suite}
Let $(a_n)_{n \ge 1}$ be a sequence of non-negative real numbers and $A$ 
a positive real number such that
$a_1 \le A$ and $a_{n+1} \le a_n \left(1- \frac{a_n}{A}\right)$. Then,
$\forall n \ge 1$,
$$a_n \le \frac{A}{n}.$$
\end{lem}

Using Lemma~\ref{lem:DVT} and Lemma~\ref{lem:suite}, it is possible to show:
\begin{theo}[{\cite[Theorem 3.6]{devore-temlyakov-96}}]
For $g \in {\mathcal L}^1$, we have
\begin{equation}\label{eq:RC}
\| g_n \| \le \|g\|^{2/3} \|g\|_{{\mathcal L}^1}^{1/3} n^{-1/6}.
\end{equation}
\end{theo}

A better rate of convergence can be proven for the Orthogonal Greedy
Algorithm. For the Orthogonal Greedy Algorithm, the following Lemma  plays 
the role of Lemma~\ref{lem:DVT}.
\begin{lem}\label{lem:DVTo}
Assume that $g \in {\mathcal L}^1$. Then, for any $n \ge 0$, $g_n^o
\in {\mathcal L}^1$ and we have:
$$\|r_{n+1}^o \otimes s_{n+1}^o\| = \frac{\langle g_{n}^o , r_{n+1}^o \otimes s_{n+1}^o \rangle}{\| r_{n+1}^o
\otimes s_{n+1}^o \|} \ge \frac{\|g_n^o\|^2}{\|g\|_{{\mathcal L}^1}}.$$
\end{lem}
\begin{proof}
Since $g_n=g - \sum_{k=1}^{n} \alpha_k r_k \otimes s_k$, it is clear 
that $g_n
\in {\mathcal L}^1$. The equality $\|r_{n+1}^o \otimes s_{n+1}^o\| 
=\frac{\langle g_{n}^o , r_{n+1}^o \otimes s_{n+1}^o \rangle}{\|
  r_{n+1}^o \otimes s_{n+1}^o \|}$ is obtained as a consequence of
the Euler-Lagrange equations associated to the optimization problem on $(r_{n+1}^o,s_{n+1}^o)$ (see~\eqref{eq:ortho2}).

 Since $g \in {\mathcal L}^1$, for any $\varepsilon >0$, we can write $g=\sum_{k \ge 0} c_k u_k
\otimes v_k$ with $\|u_k \otimes v_k\|=1$, and $\sum_{k \ge 0} |c_k| \le
\|g\|_{{\mathcal L}^1} + \varepsilon$. By~\eqref{eq:LRA_galo}, we have
$\langle g- g_n^o , g_n^o \rangle=0$, and therefore, using Lemma~\ref{lem:ProdScal}:
\begin{align*}
\|g_n^o\|^2
&= \langle g_n^o , g \rangle \\
&= \left\langle g_n^o , \sum_{k \ge 0} c_k u_k \otimes v_k \right\rangle \\
&= \sum_{k \ge 0} c_k \langle g_n^o ,u_k \otimes v_k \rangle \\
& \le \sum_{k \ge 0} |c_k| \frac{\langle g_n^o ,r_{n+1}^o \otimes s_{n+1}^o \rangle}{\| r_{n+1}^o
\otimes s_{n+1}^o \|} \\
& = (\|g\|_{{\mathcal L}^1} + \varepsilon) \frac{\langle g_n^o ,r_{n+1}^o \otimes s_{n+1}^o \rangle}{\| r_{n+1}^o
\otimes s_{n+1}^o \|},
\end{align*}
from which we conclude letting $\varepsilon$ vanish.
\end{proof}

\begin{theo}[{\cite[Theorem 3.7]{devore-temlyakov-96}}]
For $g \in {\mathcal L}^1$, we have
\begin{equation}\label{eq:RCo}
\| g_n^o \| \le \|g\|_{{\mathcal L}^1} \, n^{-1/2}.
\end{equation}
\end{theo}
\begin{proof}
We have, using~\eqref{eq:ortho} and Lemma~\ref{lem:DVTo}:
\begin{align*}
\|g_{n+1}^o \|^2
&=\bigg\|g -  \sum_{k=1}^{n+1} \alpha_k r_k^o \otimes s_k^o \bigg\|^2\\
& \le \|g_{n}^o - r_{n+1}^o \otimes s_{n+1}^o\|^2 \nonumber \\
&=\|g_{n}^o\|^2 - \|r_{n+1}^o \otimes s_{n+1}^o\|^2 \nonumber \\
&=\|g_{n}^o\|^2 \left( 1 - \frac{\|r_{n+1}^o \otimes s_{n+1}^o\|^2 }{\|g_{n}^o\|^2} \right) \\
&\le \|g_{n}^o\|^2 \left( 1 - \frac{\|g_n^o\|^2}{\|g\|_{{\mathcal L}^1}^2}
\right).
\end{align*}
The conclusion is reached applying Lemma~\ref{lem:suite} with
$a_{n}=\|g_{n-1}^o\|^2$ and $A=\|g\|_{{\mathcal L}^1}^2$.
\end{proof}

\begin{rem}
The rate of convergence of the Pure Greedy Algorithm in~\eqref{eq:RC} may be improved to $n^{-11/62}$~\cite{konyagin-temlyakov-99}. For both algorithms, it is known that there exists dictionaries and right-hand sides $f$ (even simple ones, like  a sum of only two elements of the dictionary) such that the rate of convergence $n^{-1/2}$ is attained (see~\cite{livshitz-temlyakov-03, devore-temlyakov-96,barron-cohen-dahmen-devore-08}). In that sense, the Orthogonal Greedy Algorithm realizes the
optimal rate of convergence. Notice that this rate of convergence does
not depend on the dimension of the problem. However, the assumption $g
\in {\mathcal L}^1$ seems to be more and more demanding, in terms of
regularity,  as the dimension increases (see Remark~\ref{rem:L1}).
\end{rem}

\section{Discussion and open problems}
\label{sec:discussion}
We begin this section by considering the case when  the Laplace operator is replaced by the identity
operator. We examine on this simplified case the discrepancy between the
variational
approach consisting in minimizing the energy and the non variational approach solving
the Euler-Lagrange equation.

\subsection{The Singular Value Decomposition case}\label{sec:SVD}
The algorithms we have presented above are closely related to the
Singular Value Decomposition (SVD, also called {\em rank one decomposition}). More precisely, omitting the
gradient in the optimization problem~\eqref{eq:LRA_var} yields: find $r_n \in L^2(\Omega_x)$ and $s_n \in L^2(\Omega_y)$ such that
\begin{equation}\label{eq:SVD_var}
(r_n,s_n) = \arg\min_{(r,s) \in  L^2(\Omega_x)\times L^2(\Omega_y)} \int_\Omega |g_{n-1} - r \otimes s|^2,
\end{equation}
with the recursion relation $$g_n=g_{n-1} - r_n \otimes s_n,$$ and
$g_0=g$.

In view of  the exact same arguments as in the previous sections, the
series $\sum_{n \ge 1} r_n \otimes s_n$ can be shown to converge to $g$
in $L^2(\Omega)$. This problem has a well-known companion discrete
problem, namely the SVD decomposition of a matrix (see for
example~\cite{trefethen-bau-97}). This corresponds to the case $\Omega_x=\{1, \ldots ,p\}$, $\Omega_y=\{1, \ldots ,q\}$, the integral $\int_\Omega$ is replaced by the discrete sum $\sum_{(i,j) \in {1, \ldots ,p}\times{1, \ldots ,q}}$, $G$ is a matrix in $\R^{p \times q}$ and $(R_n,S_n)$ are two (column) vectors in $\R^p \times \R^q$. In this case the tensor product $R_n \otimes S_n$ is simply the matrix $R_n (S_n)^T$. The matrices $G_n \in \R^{p \times q}$ are then defined by recursion: $G_0=G$ and $G_n=G_{n-1} - R_n (S_n)^T$.

\subsubsection{Orthogonality property}

An important property of the sequence $(r_n,s_n)$ generated by the algorithm in the SVD case is the orthogonality relation: if $n \neq m$
\begin{equation}\label{eq:SVD_ortho}
\int_{\Omega_x} r_n r_m= \int_{\Omega_y} s_n s_m=0.
\end{equation}
In order to check this, let us first write  the Euler-Lagrange equations in the SVD case (compare with~\eqref{eq:LRA_EL_FV}): for any functions $(r,s) \in L^2(\Omega_x)\times L^2(\Omega_y)$,
\begin{equation}\label{eq:SVD_EL_FV}
\int_\Omega r_n \otimes s_n (r_n \otimes s + r \otimes s_n) = \int_\Omega g_{n-1} (r_n \otimes s + r \otimes s_n).
\end{equation}
This also reads (compare with \eqref{eq:LRA_EL}):
\begin{equation}
  \left\{
      \begin{array}{l}
         \displaystyle\left(\int_{\Omega_y} |s_n|^2\right)\,  r_n = \int_{\Omega_y} 
g_{n-1}\, s_n,\\
 \\
\displaystyle\left(\int_{\Omega_x}
|r_n|^2 \right) \, s_n  =\int_{\Omega_x} g_{n-1}\, r_n.
      \end{array}
\right.
\end{equation}
It is immediate to see that \eqref{eq:SVD_EL_FV} for $n=1$ and $n=2$ 
implies,
$$\int_\Omega (r_2 \otimes s_2) (r_2 \otimes s_1) = \int_\Omega (r_2 \otimes s_2) (r_1 \otimes s_2) = 0.$$
Likewise,  it can be shown, for any $n \ge 2$ and any $l \in \{2, \ldots 
n \}$
\begin{equation}\label{eq:SVD_vrai_ortho}
\int_\Omega  \sum_{k=l}^n  (r_k \otimes s_k)
\, (r_n \otimes s_{l-1})=\int_\Omega \sum_{k=l}^n  (r_k \otimes s_k) \, (r_{l-1} \otimes s_n)=0.
\end{equation}
The orthogonality property~\eqref{eq:SVD_ortho} is then easy to check
using the Fubini Theorem and arguing by induction.

\begin{rem}
A simple consequence of the orthogonality of the functions obtained by
the algorithm is that, in the discrete version (SVD of a matrix $G \in
\R^{p \times q}$) the algorithm converges in a finite number of
iterations (namely $\max(p,q)$). As usual in this situation, practice
may significantly deviate from the above theory if round-off errors due
to floating-point computations are taken into account. This is
especially true if the matrix is ill conditioned.
\end{rem}

\subsubsection{Consequences of the orthogonality property}

The orthogonality property has several consequences:
Assume
the  SVD to be nondegenerate in the sense
\begin{equation}
  \label{eq:nondeg}
  g=\sum_{n \ge 1} \lambda_{n} \,u_n \otimes v_n,
\end{equation}
with
\begin{equation}
  \label{eq:nondeg2}
  \int_{\Omega_x} u_n u_m=\int_{\Omega_y} v_n v_m=\delta_{n,m},\forall n,m,\ \hbox{\rm and}\,
  \left(\lambda_n\right)_{n \ge 1}\hbox{\rm positive, strictly decreasing,}
\end{equation}
where $\delta_{n,m}$ is the Kronecker symbol.
Then
\begin{itemize}
\item (i) The Pure Greedy Algorithm and the Orthogonal Greedy Algorithm
  are equivalent to one another in the SVD case.
\item (ii) The SVD decomposition $g=\sum_{n \ge 1} r_n \otimes s_n$ is 
unique.
\item (iii) At iteration $n$, $\sum_{k=1}^n r_k \otimes s_k$ is the
  minimizer of $\int_\Omega |g - \sum_{k=1}^n \phi_k \otimes \psi_k|^2$ over all possible $(\phi_k,\psi_k)_{1 \le k \le n} \in \left( L^2(\Omega_x)\times L^2(\Omega_y) \right)^n$.
\end{itemize}

In addition, simple arguments show that,
\begin{itemize}
\item (iv)  The only solutions to the Euler Lagrange
  equations~\eqref{eq:SVD_EL_FV} are the null solution $(0,0)$ and  the tensor products $\lambda_n u_n
  \otimes v_n$ (for all $n \ge 1$) in the SVD decomposition of $g$.
\item  (v) The solutions to the Euler-Lagrange equations which maximize the $L^2$-norm $\left(\int_{\Omega} |r \otimes s|^2\right)^{1/2}$ are
  exactly the solutions to the  variational
  problem~\eqref{eq:SVD_var}.
\item  (vi) In dimension $N=2$, the  solutions to the Euler-Lagrange equation that satisfy the
  second order optimality conditions are  exactly the solutions of the
  original variational problem~\eqref{eq:SVD_var}.
\end{itemize}
Notice that there is no loss of generality in assuming $\lambda_n >0$, and $(\lambda_n)_{n \ge 1}$ decreasing in~ (\ref{eq:nondeg}) (up to a change of the $(u_n,v_n)$). The fundamental assumption in nondegeneracy is thus that $\lambda_n \neq \lambda_m$ if $n \neq m$. When the  decomposition 
has some degeneracy ({\em i.e.} several $n$
correspond to the same $\lambda_n$ in (\ref{eq:nondeg})) then
properties (i)-(iii)-(v)-(vi) still hold true. On the other hand, in
(ii) the SVD
is only unique up to rotations within eigenspaces and property (iv)
must be modified accordingly. In short,  the only other solutions beyond 
those mentioned above consist
of tensor products of linear combinations of functions within a given eigenspace. We skip
such technicalities. The degenerate case indeed does not differ much
from the non degenerate case above in the sense that a complete
understanding of the algorithm, both in its variational and in its non
variational forms, is at hand.

\medskip

Let us briefly outline the proofs of assertions (iv)-(v)-(vi).

\medskip




We first prove assertion (iv). It is sufficient to consider the first iteration of the algorithm. Using the SVD decomposition of $g$, the Euler-Lagrange equations write: for any functions $(r,s) \in L^2(\Omega_x)\times L^2(\Omega_y)$,
\begin{equation*}
\int_\Omega r_1 \otimes s_1 (r_1 \otimes s + r \otimes s_1) = \sum_{n \ge 1} \lambda_n \int_\Omega u_n \otimes v_n (r \otimes s_1 + r_1 \otimes 
s).
\end{equation*}
Using the orthogonality property, and successively $(r,s)=(0,v_n)$ and $(r,s)=(u_n,0)$  as test functions, we get
\begin{equation*}
\left\{
\begin{array}{l}
\displaystyle \int_{\Omega_x} |r_1|^2 \int_{\Omega_y} s_1 v_n = \lambda_n \int_{\Omega_x} r_1 u_n,\\
\\
\displaystyle \int_{\Omega_y} |s_1|^2 \int_{\Omega_x} r_1 u_n = \lambda_n \int_{\Omega_y} s_1 v_n,
\end{array}
\right.
\end{equation*}
which yields: $\forall n \ge 1$
$$\int_{\Omega_y} s_1 v_n \int_{\Omega_x} r_1 u_n \left(\int_{\Omega_x} |r_1|^2 \int_{\Omega_y} |s_1|^2 - (\lambda_n)^2 \right)=0.$$
Since for $n \neq m$, $\lambda_n \neq \lambda_m$, this shows that either $r_1 \otimes s_1 = 0$, or there exists a unique $n_0$ such that $\lambda_{n_0}  = \sqrt{ \int_{\Omega} |r_1 \otimes s_1|^2}$ and $\forall n \neq n_0$, $\int_{\Omega_y} s_1 v_n =  \int_{\Omega_x} r_1 u_n=0$ (because the product $\int_{\Omega_y} s_1 v_n  \int_{\Omega_x} r_1 u_n$ cancels 
and thus each of the term cancels because of the Euler Lagrange equations). Since by the Euler-Lagrange equations, $r_1$ (resp. $s_1$) can be decomposed on the set of orthogonal functions $(u_n,\, n \ge 1)$ (resp. $(v_n,\, n \ge 1)$), we get $r_1 \otimes s_1 = \lambda_{n_0} u_{n_0} \otimes v_{n_0}$, which concludes the proof of assertion (iv). Assertion (v) is 
readily obtained using (iv) and the orthogonality property. Notice that assertion (ii) is a consequence of assertions (iv)-(v). To prove assertion (vi), we recall that the second order optimality condition writes (see Lemma~\ref{lem:EL2}, adapted to the SVD case): $\forall (r,s) \in L^2(\Omega_x) \times L^2(\Omega_y)$,
\begin{equation}\label{eq:EL2_SVD}
\left(\int_\Omega ( r_n \otimes s_n - g_{n})  r \otimes s \right)^2 \le \int_\Omega | r
  \otimes s_n|^2\ \int_\Omega |r_n \otimes s|^2.
\end{equation}
It is again enough to consider the case $n=1$. Let us consider a solution of the Euler-Lagrange equation: $r_1 \otimes s_1 = \lambda_{n_0} u_{n_0} \otimes v_{n_0}$, and let us take as test functions in~\eqref{eq:EL2_SVD} $(r,s)=(u_n,v_n)$, for all $n \ge 1$. We obtain that for all $n \ge 1$, $(\lambda_{n})^2 \le (\lambda_{n_0})^2$ which concludes the proof 
of assertion (vi). Notice that in dimension $N \ge 3$, 
assertion (vi) seemingly  does not hold: the solutions to the Euler-Lagrange
equation that satisfy the second order optimality conditions may not
necessarily be global minimizers.

\subsubsection{Link between the Euler-Lagrange equations and the variational problem}

Properties (iv)-(v)-(vi) above tend to indicate that, at least in the
SVD case, the consideration of the solutions to the
Euler-Lagrange equations is somehow close to the consideration of the
minimization problems.

Indeed, if we assume that at each iteration, non zero solutions of the Euler-Lagrange equations are obtained (of course under the assumption~$g_{n-1} \neq 0$ in~\eqref{eq:SVD_EL_FV}), then the non variational
form of the algorithm, if it converges, 
will eventually provide the correct
decomposition. We however would like to mention two practical
difficulties.

First, it is not clear in practice how to compute the norm  $\|g_n\|$ to check the convergence, since this is in general a high dimensional integral. A more realistic convergence criterion would read: $\|r_n \otimes s_n\|$ 
{\it is small compared to} $\left\|\sum_{k=1}^{n-1} r_k \otimes
  s_k\right\|$. However, using this criterion, it is possible to
erroneously conclude that the algorithm has converged, while a term with
an arbitrarily large contribution has been missed. Indeed, consider again, to 
convey the idea, the case
(\ref{eq:nondeg})-(\ref{eq:nondeg2}). Assume that  the tensor product
$\lambda_2u_2\otimes v_2$
is picked at first iteration (\emph{instead of} the tensor product
$\lambda_1u_1\otimes v_1$ which would be selected by the \emph{variational} version
of the algorithm). Assume similarly that $\lambda_3u_3\otimes v_3$  is
picked at second iteration, and so on and so forth. In such a situation, 
one would then
decide the series $\displaystyle\sum_{n\geq 2}\lambda_nu_n\otimes v_n$
solves the problem, while obsviously it does not. We will show below
(see Section~\ref{sec:resol-EL}) that in the simple fixed-point
procedure we have described above to solve the nonlinear Euler-Lagrange
equations, the fact that $\lambda_1 u_1 \otimes v_1$ is missed, and
never obtained as a solution, may indeed happen as soon as the initial
condition of the iterative procedure has a zero component on the
eigenspace  associated to $\lambda_1$.

Second, without
an additional assumption 
reminiscent of the minimizing character of the solution,
iteratively solving the Euler-Lagrange equations may result in picking the
 tensor products   $\lambda_n u_n
  \otimes v_n$ in an order not appropriate for computational
  efficiency. Such an assumption is present in assertions (v) and (vi). For the illustration, let us indeed consider a SVD 
decomposition
$$g=\sum_{n\ge 1}\,\lambda_n u_n\otimes v_n$$
for some functions $u_n$ and $v_n$ that become highly oscillatory when
$n$ grows. It is clear that  we may
obtain an error in $H^1$ norm that is arbitrarily large at each
iteration of the algorithm. In particular, it may happen (in particular
if smooth functions are chosen as initial guesses for the nonlinear
iteration loop solving the Euler-Lagrange equation) that the highly
oscillatory products are only selected in the latest iterations, although they
 contribute  to the error in a major way. A poor efficiency of the
 algorithm follows. Inevitably, reaching computational efficiency therefore requires
 to account for some
 additional assumptions to select the appropriate solutions among the
 many solutions of the Euler-Lagrange equations.

\medskip

In the spirit of the above discussion, one can notice that
\begin{itemize}
\item (vii) The null solution
$(0,0)$ to the Euler-Lagrange equation~(\ref{eq:SVD_EL_FV}) is generically not isolated within the set of all solutions.
\end{itemize} Indeed, consider a SVD  $\displaystyle
g=\sum_{n \ge 1} \lambda_n \,u_n\otimes v_n$, such that $u_n$ and $v_n$ are non-zero functions for all
$n\ge 1$ (and $\lambda_n >0$). Then,  any $(\lambda_n u_n,v_n)$ is a
solution of the Euler Lagrange equation at the first iteration, and the
norm of the $(\lambda_n u_n,v_n)$ which is selected may be arbitrarily small since the series
$\sum_{n \ge 1} \lambda_n \,u_n\otimes v_n$ converges, and therefore $\| 
\lambda_n u_n\otimes v_n \|$ goes to zero. A similar argument
applies to all iterations of the algorithm. Therefore, a criterion of
convergence of the type  $\|r_n \otimes s_n\|$ {\it is small compared
  to} $\left\|\sum_{k=1}^{n-1} r_k \otimes s_k\right\|$ may again yield
an erroneous conclusion and lead to a  prematurate termination of the iterations.


\begin{rem}
   Note of course that the relaxation step performed in the orthogonal
  version of the algorithm does not solve any of the above difficulties.
\end{rem}

\subsubsection{Resolution of the Euler-Lagrange equations}\label{sec:resol-EL}

A last comment we would like to make on the SVD case again concerns the
practical implementation of the solution procedure for the Euler-Lagrange
equations. Consider  the discrete case for clarity. The fixed point
procedure then simply writes (for a fixed $n$): at iteration $k \ge 0$, compute two vectors $(R_n^k,S_n^k) \in \R^p \times \R^q$ such that:
\begin{equation}\label{eq:SVD_FP}
\left\{
\begin{array}{l}
(S_n^k)^T S_n^k  R_n^{k+1} = G_{n-1} S_n^k,\\[4pt]
(R_n^{k+1})^T R_n^{k+1} S_n^{k+1} = (G_{n-1})^T R_n^{k+1}.
\end{array}
\right.
\end{equation}
One can check that this procedure is similar to the power method to
compute the largest eigenvalues (and associated eigenvectors) of the
matrix $(G_{n-1})^T G_{n-1}$. Let us explain this. The recursion writes:
$$S^{k+1}=(G^T G) S^k \frac{\|S^k\|^2}{\|GS^k\|^2},$$
where $\|\cdot\|$ here denotes the Euclidean norm and where we have
omitted the subscripts $n$ and $n-1$ for clarity. To study the convergence of this algorithm one can assume that $G$ is actually a diagonal matrix up to a change of coordinate. Indeed, let us introduce the SVD decomposition of $G$: $G=U \Sigma V^T$ where $U$ and $V$ are two orthogonal matrices, and $\Sigma$ is a diagonal matrix with non-negative coefficients. 
Without loss of generality, we may assume that $q \le p$, $U \in
\R^{p\times q}$, $\Sigma \in \R^{q\times q}$, $V \in \R^{q\times q}$ and
$\Sigma_{1,1} \ge \Sigma_{2,2} \ge \ldots \ge \Sigma_{q,q}$. For
simplicity, assume that $\Sigma_{1,1} > \Sigma_{2,2} > 0$. Then, setting
$\tilde{S}^k=V^TS^k$, the recursion reads $\tilde{S}^{k+1}=(\Sigma^T \Sigma) \tilde{S}^k \frac{\|\tilde{S}^k\|^2}{\|\Sigma 
\tilde{S}^k\|^2}$ and the convergence is easy to study. One can check that if the initial condition $S^0$ has a non-zero component along the vector associated to the largest value $\Sigma_{1,1}$, then $S^k$ converges to this vector. The convergence is geometric, with a rate related to $\frac{\Sigma_{2,2}}{\Sigma_{1,1}}$ (at least if the initial condition $S^0$ has a non-zero component along the vector associated to $\Sigma_{2,2}$, otherwise $\Sigma_{2,2}$ should be replaced by the appropriate largest $\Sigma_{k,k}$, with $k > 1$). Of course, if the initial condition is not well chosen (namely, if $S^0$ has a zero component along the vector associated to $\Sigma_{1,1}$), then this algorithm cannot converge to the solution of the variational version of the algorithm.

We would like to mention that this method to compute the SVD of a matrix 
is actually known to poorly perform in practice. More precisely, the
approach is very sensitive to numerical perturbations, see~\cite[Lecture
31]{trefethen-bau-97}) since the condition number of $(G_{n-1})^T
G_{n-1}$ is typically large. Alternative methods exist that compute the 
SVD decomposition, and it would be interesting to use these techniques as guidelines to build more efficient procedures to solve the nonlinear Euler-Lagrange equations~\eqref{eq:LRA_EL}.

\subsection{Euler-Lagrange approach for the Poisson problem}

We now  return to the
solution of the Poisson problem. Our purpose is to see which of the
above mentioned difficulties survive in this case. We shall also see new
difficulties appear.

\medskip

We first observe, on a general note,  that a  property similar to~\eqref{eq:SVD_vrai_ortho} holds in the Poisson case, namely:
\begin{equation}\label{eq:LRA_vrai_ortho}
\int_\Omega  \nabla \left(  \sum_{k=l}^n  r_k \otimes s_k  \right) \cdot
 \nabla (r_n \otimes s_{l-1})=\int_\Omega \nabla \left( \sum_{k=l}^n 
 r_k \otimes s_k \right) \cdot \nabla (r_{l-1} \otimes s_n)=0.
\end{equation}
This, however, does not seem to imply any simple orthogonality property
as~\eqref{eq:SVD_ortho}. In particular, in the Poisson case, it is
generally wrong that, for $n \neq m$, $\int_{\Omega_x} \nabla (r_n
\otimes s_n) \cdot  \nabla (r_m \otimes s_m)=0$.

Next, we remark that none of the properties (i)-(ii)-(iii) holds in the
Poisson case. Likewise, we are not able to characterize the list of
solutions to the Euler-Lagrange equations as we did in (iv)-(v)-(vi).

\medskip

This is for the generic situation, but in order to better demonstrate the connections
between the SVD case above and the  Poisson case,
let us show that, in fact, the Poisson case necessarily embeds all the
difficulties of the SVD case. For this purpose, we consider the original
algorithm (for the Poisson problem) performed for a particular right-hand-side $f=-\Delta g$, namely
\begin{equation}\label{eq:hyp_ortho_g}
\left\{
\begin{array}{l}
\text{$g=\sum_{k=1}^N \alpha_k \phi_k \otimes \psi_k$ where $\alpha_k
  \in \R$,}\\[4pt]
\text{$\phi_k$ (resp. $\psi_k$) are eigenfunctions of}\\[4pt]
\text{the homogeneous Dirichlet operator
  $-\partial_{xx}$ (resp. $-\partial_{yy}$)}\\[4pt]
\text{and satisfy $\forall k, l, \, \int \phi_k \phi_l = \int
  \psi_k \psi_l = \delta_{k,l}$,}
\end{array}
\right.
\end{equation}
where $\delta_{k,l}$ is again the Kronecker symbol. Then, it can be
shown that, as in the SVD case, $r_k \otimes s_k = \alpha_k \phi_k
\otimes \psi_k$ are indeed solution to the Euler-Lagrange
equations~\eqref{eq:LRA_EL_FV}. This suffices to show the non uniqueness 
of
the solution. Furthermore, and in sharp contrast to~(iv),  there even exist
solutions to the Euler Lagrange equations that are not of the above
form.

Here is an example of the latter claim. Consider the case $\phi_1=\psi_1$, associated with an eigenvalue $\lambda_1$ and $\phi_2=\psi_2$, associated with an eigenvalue $\lambda_2\neq \lambda_1$. We suppose $\alpha_k=0$ for $k \ge 2$. We are looking for
$r$ and $s$ solution to the Euler-Lagrange equations
\begin{equation*}
\left\{
\begin{array}{l}
\displaystyle - \int |s|^2   r''   + \int
|s'|^2  r  =\int f  s,\\
\\
\displaystyle- \int |r|^2   s''  + \int
|r'|^2   s  = \int f r.
\end{array}
\right.
\end{equation*}
Then, it can be checked that $r=r_1 \phi_1 + r_2 \phi_2$ and $s=s_1 \psi_1 + s_2 \psi_2$ are solution to the Euler-Lagrange equations, with the following set of parameters: $r_1=1$, $r_2=1/2$, $s_1=2$, $s_2=1$, $\alpha_1=\frac{9 \lambda_1 + \lambda_2}{4
  \lambda_1}$ and $\alpha_2=\frac{2 \lambda_1 + 3 \lambda_2}{2
  \lambda_2}$. Likewise, it is immediate to see that (vii) still holds.
In view of the above remarks, it seems difficult to devise (and, even
more difficult, to prove the convergence of)  efficient iterative
procedures to correctly solve the Euler-Lagrange equation.


\subsection{Some numerical experiments and the non self-adjoint case}

We now show some numerical tests. Even though the algorithms presented above have been designed for
solving problems in high dimension, we restrict ourselves to the
two-dimensional case. For numerical results in higher dimension, we refer to~\cite{ammar-mokdad-chinesta-keunings-06}. Moreover, we consider the discrete case mentioned in Section~\ref{sec:SVD}, which writes (compare with~\eqref{eq:lapl}): for a given symmetric positive definite matrix  $D \in \R^{d \times d}$ (which plays the role of the one-dimensional operator $-\partial_{xx}$), and a given matrix $F \in \R^{d \times  d}$ (which plays the role of the right-hand side $f$):
\begin{equation}\label{eq:PbCont}
\text{Find $G \in \R^{d \times d}$ such that } DG + GD = F.
\end{equation}
Here, the dimension $d$ typically corresponds to the number of points used
 to discretize the one-dimensional functions $r_n$ or $s_n$. To this problem is associated the variational problem (compare with~\eqref{eq:lapl_var})
\begin{equation}\label{eq:FV}
\text{Find $G \in \R^{d \times d}$ such that } G= \arg\min_{U \in \R^{d
 \times d}} \left(\frac{DU+UD}{2} -F \right):U,
\end{equation}
where, for two matrices $A,B \in \R^{d \times d}$, $A:B = \sum_{1 \le i,j \le d} A_{i,j} B_{i,j}$. The matrix $G$ is built as a sum of rank one 
matrices $R_k S_k^T$ with $(R_k, S_k) \in (\R^{d})^2$, using the
following Pure Greedy Algorithm (compare with the algorithm presented in
Section~\ref{sec:algo}): {\tt Set $F_0=F$ and at iteration $n \ge 1$,
\begin{enumerate}
\item Find $R_n$ and $S_n$ two vectors in $\R^d$ such that:
\begin{equation}\label{eq:FV_LRA}
(R_n,S_n)= \arg\min_{(R,S) \in (\R^{d})^2} \left(\frac{D (RS^T) +(RS^T)
    D}{2} -F_{n-1} \right):(RS^T).
\end{equation}
\item Set\footnote{In practice, to avoid numerical cancellation, we actually set $F_{n}=F-(D U_n+U_n D)$ where $U_n=\sum_{k=1}^n R_k S_k^T$.} $F_{n}=F_{n-1}-(D R_n S_n^T+ R_n S_n^T D)$.
\item If $\|F_{n}\|>\varepsilon$, proceed to iteration $n+1$. Otherwise stop.
\end{enumerate}
}
As explained in Section~\ref{sec:EL}, Step 1 of the above algorithm
is replaced in practice by the resolution of the associated
Euler-Lagrange equations. This consists in finding  two vectors $R_n$
and $S_n$ in $\R^d$ solution to the nonlinear equations:
\begin{equation}\label{eq:EL_LRA}
\left\{
\begin{array}{l}
\|S_n\|^2 \, D R_n + \|S_n\|_D^2 \, R_n  =F_{n-1} S_n, \\[5pt]
\|R_n\|^2 \, D S_n + \|R_n\|_D^2 \, S_n  =F_{n-1}^T R_n,
\end{array}
\right.
\end{equation}
where, for any vectors $R \in \R^d$, we set $\|R\|_D^2=R^T D R$. This
nonlinear problem is solved by a simple fixed point procedure
(as~\eqref{eq:LRA_FP}). We have observed in practice that choosing a
random vector as an initial condition for the fixed point procedure is
more efficient than taking a given deterministic vector (like $(1,\ldots,1)^T$). This is of course 
related to the convergence properties of the fixed point procedure we discussed in Section~\ref{sec:resol-EL}.

\subsubsection{Convergence of the method}

In this section, we take $D$ diagonal, with $(1,2,\ldots,d)$ on the
diagonal, and a random matrix $F$. The parameter $\varepsilon$ is
$10^{-6}$. We observe that the algorithm always converges. This means that, in practice, the solutions of the Euler-Lagrange equations~\eqref{eq:EL_LRA} 
selected by the fixed point procedure are appropriate.

On Figure~\ref{fig:NRJ}, we plot the energy $\left(\frac{DU_n + U_nD}{2}
  -F\right) : U_n$, where $U_n=\sum_{k=1}^n R_k S_k^T$. We observe
that the energy rapidly decreases and next reaches a plateau. This is a
general feature  that we observe on all the tests we perfomed.
\begin{figure}[htbp]
\begin{center}
\includegraphics[width=10cm,angle=270]{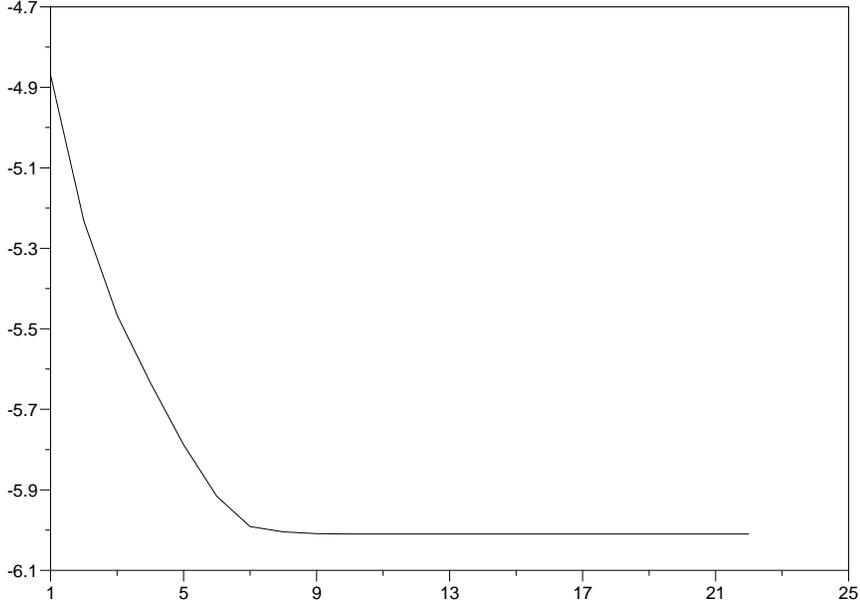}
\end{center}
\caption{Evolution of the energy as a function of iterations ($d=10$, $D=diag([linspace(1,2,d)])$, $\varepsilon=10^{-6}$).}\label{fig:NRJ}
\end{figure}

In Table~\ref{tab:d}, we give the number of iterations necessary for
convergence, as a function of $d$. We observe a linear dependency, which
unfortunately we are unable to explain theoretically.
\begin{table}[htbp]
\begin{center}
\begin{tabular}{|c|ccc|}
\hline
$d$ & 10 & 20 & 30  \\
\hline
Number of iterations & 22-23 & 45-46 & 69-70  \\
\hline
\end{tabular}
\end{center}
\caption{Number of iterations typically needed for convergence as a function of $d$, for various random matrices $F$ ($D=diag([linspace(1,2,d)])$,
$\varepsilon=10^{-6}$).}\label{tab:d}
\end{table}

\subsubsection{The non self-adjoint case}

In~\cite{ammar-mokdad-chinesta-keunings-06}, it is actually proposed to use the Orthogonal Greedy Algorithm for non self adjoint operators.

Consider, for the prototypical case of an advection diffusion equation:
\begin{equation}\label{eq:adv_diff}
\text{Find $g \in H^1_0(\Omega) $ such that }\left\{
\begin{array}{rl}
a \cdot \nabla g -\Delta g  =f &\text{ in $\Omega$},\\
g=0& \text{ on $\partial \Omega$},
\end{array}
\right.
\end{equation}
where $a : \Omega \to \R^2$ is a given smooth velocity field. When $a=
\nabla V$ for some real-valued function $V$, problem
(\ref{eq:adv_diff}) is equivalent to minimizing the energy 
$$
  \frac{1}{2} \int_\Omega |\nabla u|^2 \exp(-V) - \int f u  \exp(-V).$$
 When this is not the case, it is not in general possible to
 recast~\eqref{eq:adv_diff} in terms of  a minimization problem. However, a variational formulation can be written as: Find $g \in H^1_0(\Omega)$ such that, for all $v \in H^1_0(\Omega)$, $$\int_\Omega (a \cdot \nabla g)  v +\nabla g \cdot \nabla v=  \int_\Omega f v.$$
It is proposed in~\cite{ammar-mokdad-chinesta-keunings-06} to use this
variational formulation 
 in step 1 and 2 of the Orthogonal Greedy Algorithm. The iterations 
 then write:
{\tt set $f_0=f$, and at iteration $n \ge 1$,
\begin{enumerate}
\item Find $r_n \in H^1_0(\Omega_x)$ and $s_n \in H^1_0(\Omega_y)$ such
  that, for all functions $(r,s) \in H^1_0(\Omega_x) \times H^1_0(\Omega_y)$,
\begin{equation}\label{eq:advdiff_EL}
\int_\Omega (a \cdot \nabla (r_n \otimes s_n))  (r_n \otimes s + r \otimes s_n) +\nabla (r_n \otimes s_n) \cdot \nabla (r_n \otimes s + r \otimes 
s_n)=  \int_\Omega f_{n-1} (r_n \otimes s + r \otimes s_n).
\end{equation}
\item Find $u_n \in {\rm Vect}(r_1 \otimes s_1, \ldots, r_n \otimes s_n)$ such that for all $v \in {\rm Vect}(r_1 \otimes s_1, \ldots, r_n \otimes s_n)$
\begin{equation}\label{eq:advdiff_gal}
\int_\Omega (a \cdot \nabla u_n)  v +\nabla u_n \cdot \nabla v  =\int_\Omega f v.
\end{equation}
\item Set $f_{n}=f_{n-1} - (a \cdot \nabla u_n -\Delta u_n)$.
\item If $\|f_{n}\|_{H^-1(\Omega)} \ge \varepsilon$, proceed to iteration $n+1$. Otherwise, stop.
\end{enumerate}
}
The corresponding discrete formulation reads: 
\begin{equation}\label{eq:PbCont_NA}
\text{Find $G \in \R^{d \times d}$ such that } B G + G B^T = F,
\end{equation}
where $B$ is not supposed to be symmetric here (compare
to~\eqref{eq:PbCont}). The numerical method reads: {\tt  Set $F_0=F$ and at iteration $n \ge 
1$,
\begin{enumerate}
\item Find $R_n$ and $S_n$ two vectors in $\R^d$ such that:
\begin{equation}\label{eq:EL_LRA_NA}
\left\{
\begin{array}{l}
\|S_n\|^2 \, B R_n + \|S_n\|_B^2 \, R_n = F_{n-1} S_n, \\[5pt]
\|R_n\|^2 \, B S_n + \|R_n\|_B^2 \, S_n = F_{n-1}^T R_n.
\end{array}
\right.
\end{equation}
\item Set $F_{n}=F_{n-1}-(B R_n S_n^T+ R_n S_n^T B^T)$.
\item If $\|F_{n}\|>\varepsilon$, proceed to iteration $n+1$. Otherwise stop.
\end{enumerate}
}
We consider the case when $B=  D + A$ with $D$ symmetric positive definite, and $A$ antisymmetric, so that we know there exists a unique solution to~\eqref{eq:PbCont_NA}. On the numerical tests we have performed, the 
algorithm seems to  converge. In the absence of any energy minimization
principle, it is however unclear to us
how to prove convergence of this algorithm. 



\end{document}